\documentclass[12 pt]{amsart}
\usepackage{amscd,amsmath,amsthm,amssymb}
\usepackage{verbatim}
\usepackage[all]{xy}
\usepackage{color}
\usepackage{hyperref}
\usepackage{multirow}
\usepackage{subcaption}

\usepackage{tikz, float} \usetikzlibrary {positioning,calc,intersections}

\usepackage{enumitem}

\let\chapter\undefined

\def\NZQ{\mathbb}               
\def\NN{{\NZQ N}}

\def\FF{{\NZQ F}}

\def\frk{\mathfrak}               

\def\Phi{{\frk n}}
\def\Phi{{\frk N}}

\def\kb{{\mathbf k}}

\def\xb{{\mathbf x}}

\def\A{{\mathcal A}}

\def\P{{\mathcal P}}

\def\C{{\mathcal C}}

\def\xb{{\mathbf x}}

\def\opn#1#2{\def#1{\operatorname{#2}}} 
\opn\chara{char} \opn\length{\ell} \opn\pd{pd} \opn\rk{rk}
\opn\projdim{proj\,dim} \opn\injdim{inj\,dim} \opn\rank{rank}
\opn\depth{depth} \opn\grade{grade} \opn\height{height}
\opn\embdim{emb\,dim} \opn\codim{codim}

\opn\Tr{Tr} \opn\bigrank{big\,rank}
\opn\superheight{superheight}\opn\lcm{lcm}
\opn\trdeg{tr\,deg}
\opn\reg{reg} \opn\lreg{lreg} \opn\ini{in} \opn\lpd{lpd}
\opn\size{size} \opn\sdepth{sdepth}
\opn\link{link}\opn\fdepth{fdepth}\opn\lex{lex}
\opn\LM{LM}
\opn\LC{LC}
\opn\NF{NF}
\opn\Merge{Merge}
\opn\sgn{sgn}
\opn\suppPos{suppPos}
\opn\div{div} \opn\Div{Div} \opn\cl{cl} \opn\Pic{Pic}
\opn\Prin{Prin}
\opn\op{op}
\opn\indeg{indeg} \opn\outdeg{outdeg}
\opn\red{red}
\opn\Spec{Spec} \opn\Supp{Supp} \opn\supp{supp} \opn\Sing{Sing}
\opn\Ass{Ass} \opn\Min{Min}\opn\Mon{Mon}
\opn\Ann{Ann} \opn\Rad{Rad} \opn\Soc{Soc}
 \opn\Ker{Ker} \opn\Coker{Coker} \opn\Am{Am}
\opn\Hom{Hom} \opn\Tor{Tor} \opn\Ext{Ext} \opn\End{End}
\opn\Aut{Aut} \opn\id{id}

\opn\nat{nat}
\opn\pff{pf}
\opn\Pf{Pf} \opn\GL{GL} \opn\SL{SL} \opn\mod{mod} \opn\ord{ord}
\opn\Gin{Gin} \opn\Hilb{Hilb}\opn\sort{sort}
\opn\span{span}
\opn\Image{Image}
\opn\aff{aff} \opn\con{conv} \opn\relint{relint} \opn\st{st}
\opn\lk{lk} \opn\cn{cn} \opn\core{core} \opn\vol{vol}
\opn\link{link} \opn\star{star}\opn\lex{lex}\opn\set{set}
\opn\dist{dist}
\opn\gr{gr}

\def\pot#1#2{#1[\kern-0.28ex[#2]\kern-0.28ex]}

\opn\dirlim{\underrightarrow{\lim}}
\opn\inivlim{\underleftarrow{\lim}}

\def\MVT{{\rm MVT}}

\let\to=\rightarrow

\def\Implies{\ifmmode\Longrightarrow \else
        \unskip${}\Longrightarrow{}$\ignorespaces\fi}
\def\implies{\ifmmode\Rightarrow \else
        \unskip${}\Rightarrow{}$\ignorespaces\fi}
\def\iff{\ifmmode\Longleftrightarrow \else
        \unskip${}\Longleftrightarrow{}$\ignorespaces\fi}

\let\:=\colon
\newtheorem{Theorem}{Theorem}[section]

\newtheorem{Corollary}[Theorem]{Corollary}
\newtheorem{Proposition}[Theorem]{Proposition}

\newtheorem{Question}[Theorem]{Question}
\theoremstyle{remark}
\newtheorem{Remark}[Theorem]{Remark}

\theoremstyle{definition}
\newtheorem{Example}[Theorem]{Example}

\newtheorem{Definition}[Theorem]{Definition}

\let\kappa=\varkappa
\def\qed{\ifhmode\textqed\fi
      \ifmmode\ifinner\quad\qedsymbol\else\dispqed\fi\fi}
\def\textqed{\unskip\nobreak\penalty50
       \hskip2em\hbox{}\nobreak\hfil\qedsymbol
       \parfillskip=0pt \finalhyphendemerits=0}
\def\dispqed{\rlap{\qquad\qedsymbol}}

\opn\dis{dis}
\def\pnt{{\raise0.5mm\hbox{\large\bf.}}}

\opn\Lex{Lex}
\opn\syz{{\rm syz}}
\opn\spoly{{\rm spoly}}
\opn\LM{{\rm LM}}
\opn\lm{{\rm lm}}
\opn\projdim{{\rm projdim}}
\opn\lcm{{\rm lcm}} \opn\A{\mathcal A}

\opn\prob{{\rm prob}}

\numberwithin{equation}{section}

\tikzstyle{Cwhite}=[scale = .6,circle, fill = white, minimum size=2.5mm]
\tikzstyle{Cgray}=[scale = .4,circle, fill = gray, minimum size=3mm]
\tikzstyle{Cblack2}=[scale = .4,circle, fill = black, minimum size=3mm]
\tikzstyle{Cblack}=[scale = .7,circle, fill = black, minimum size=3mm]
\tikzstyle{C0}=[scale = .9,circle, fill = black!0, inner sep = 0pt, minimum size=3mm]
\tikzstyle{C1}=[scale = .7,circle, fill = black!0, inner sep = 0pt, minimum size=3mm]
\tikzstyle{Cred}=[scale = .4,circle, fill = red, minimum size=3mm]
\tikzstyle{Cblue}=[scale = .4,circle, fill =blue, minimum size=3mm]

\begin{document}

\title{Support posets of some monomial ideals}

\author{Patricia Pascual-Ortigosa} 
\address{Departamento de Matem\'aticas y Computaci\'on,  Universidad de La Rioja, Spain}
\email{papasco@unirioja.es}

\author{Eduardo S\'aenz-de-Cabez\'on}
\address{Departamento de Matem\'aticas y Computaci\'on,  Universidad de La Rioja, Spain}
\email{eduardo.saenz-de-cabezon@unirioja.es}

\maketitle
\begin{abstract}
The support poset of a monomial ideal $I\subseteq\kb[x_1,\dots,x_n]$ encodes the relation between the variables $x_1,\dots,x_n$ and the minimal monomial generators of $I$. It is known that not every poset is realizable as the support poset of some monomial ideal. We describe some posets $P$ for which we can explicitly find at least one monomial ideal $I_P$ such that $P$ is the support poset of $I_P$. Also, for some families of monomial ideals we describe their support posets and study their properties. As an example of application we examine the relation between forests and series-parallel ideals.

\end{abstract}
\section{Introduction}\label{sec:intro}
The support poset of a monomial ideal $I\subseteq\kb[x_1,\dots,x_n]$ encodes the relation between the variables $x_1,\dots,x_n$ and the minimal monomial generators of $I$. It was introduced in \cite{MPSW19} to study the set of all depolarizations of a given squarefree monomial ideal. Since many relevant features of a given monomial ideal are shared by the ideals in the same polarity class, the study of polarization and depolarization has become relevant in the last years cf. \cite{F05,HH10,HM10,IKM15,J07}. In this context, the use of the support poset is a useful tool.
 
As it is shown in \cite{MPSW19} not every poset is realizable as the support poset of a monomial ideal. A natural problem is therefore to find posets that can be realized as support posets of monomial ideals and provide explicit descriptions of those ideals, so that we can describe properties of the ideal based on properties of the support poset and viceversa. We address this issue in Section \ref{sec:givenPoset} of the paper in which we give some families of posets for which we can always find at least one monomial ideal supported by them (collections of lines or diamonds, and forests) and provide a full explicit description of the main features of these ideals such as their Betti numbers and free resolutions, see Propositions \ref{prop:Betti-lines} and \ref{prop:diamondsBetti} and in particular Theorem \ref{th:leafIdeal}.

Another natural question related to support posets is to find a natural way to describe the support poset of some families of monomial ideals. We describe in Section \ref{sec:givenIdeal} the support poset of $k$-out-of-$n$ and series-parallel ideals, which correspond to relevant systems in reliability theory \cite{KZ03,SW10,SW11}. We find a particular relation between forests and series-parallel ideals, see Theorem \ref{th:spIdeal} and Proposition \ref{prop:spIdeal}. It is known that a given poset can be the support poset of several different monomial ideals. We see that this holds even within the classes of forests and series-parallel ideals, i.e. a given forest can be the support poset of several different series-parallel ideals.

We finish the paper with several open questions on support posets.

The paper starts with a section containing the basic definitions and results on support posets. The second part of that section recovers the main notions on Mayer-Vietoris trees \cite{S09}, which will be used in the proofs of Section \ref{sec:givenPoset}.

\section{Preliminaries and basic notions on support posets and Mayer-Vietoris trees}
\subsection{The support poset}
Let $R=\kb[x_1,\dots,x_n]$ be a polynomial ring in $n$ variables. For any monomial $m$ of $R$ the {\em support} of $m$, denoted by $\supp(m)$, is defined as the set of indices of variables which divide $m$. The support of a monomial ideal $I\subseteq R$ is $\supp(I)=\bigcup_{m\in G(I)}\supp(m)$, where $G(I)$ is the unique minimal monomial generating set of $I$. We say that an ideal $I$ has {\em full support} if $\supp(I)=\{1,\dots,n\}=[n]$.  For ease of notation we assume that ideals have full support, unless otherwise stated.

Let $I$ be a squarefree monomial ideal with $G(I)=\{m_1,\dots,m_r\}$. 
For each $i$ in $\supp(I)$ we define the set $C_i\subseteq \supp(I)$ as,
\[
C_i=\{  j\,\vert \, j\in \bigcap_{m\in G(I)}\{\supp(m)\vert x_i \mbox{ divides } m\} \}.
\]
In other words, $C_i$ is given by the indices of all the variables that appear in every minimal generator of $I$ in which $x_i$ is present.
Let $C_I=\{C_1,\ldots,C_n\}$. The poset on the elements of $C_I$ ordered by inclusion is called the {\em support poset} of $I$ and is denoted $\suppPos(I)$.

\begin{Definition}\label{def:polarization}
Let $a=(a_1,\dots,a_n)$ and $\mu=(b_1,\dots,b_n)$ be two elements in $\NN^n$ with $b_i\leq a_i$ for all $i$. The polarization of $\mu$ in $\NN^{a_1+\cdots+a_n}$ is the multi-index $$\overline{\mu}=(\underbrace{1,\dots,1}_{b_1},\underbrace{0,\dots,0}_{a_1-b_1},\dots,\underbrace{1,\dots,1}_{b_n},\underbrace{0,\dots,0}_{a_n-b_n}).$$ The {\em polarization of} $\xb^\mu=x_1^{b_1}\cdots x_n^{b_n}\in R$ with respect to $a$ is the squarefree monomial $\xb^{\overline{\mu}}=x_{1,1}\cdots x_{1,b_1}\cdots x_{n,1}\cdots x_{n,b_n}$ in $S=\kb[x_{1,1},\dots,x_{1,a_1},\dots,x_{n,1},\dots,x_{n,a_n}]$. Note that for ease of notation we used $\xb$ with two different meanings in this definition.
Let $I=\langle m_1,\dots,m_r\rangle\subseteq R$ be a monomial ideal and let $a_i$ be the maximum exponent to which indeterminate $x_i$ appears among the generators of $I$.
The {\em polarization of $I$}, denoted by $I^P$, is the monomial ideal in $S$ given by $I^P=\langle \overline{m_1},\dots,\overline{m_r}\rangle$, where $\overline{m_i}$ is the polarization of $m_i$ with respect to $a$. 
\end{Definition}

\begin{Definition}
Let $R, S$ and $T$ be polynomial rings over the field $\kb$. Let $I\subseteq R$ be a squarefree monomial ideal. A {\em depolarization of $I$} is a monomial ideal $J\subseteq S$ such that $I$ is isomorphic to $J^P\subseteq T$ that is:
There is a bijective map $\varphi$ from the set of variables of $R$ to the set of variables of $T$ such that $\varphi(G(I))=G(J^P)$, 
 where $G(J^P)$ is the unique minimal monomial generating set of $J^P$.
\end{Definition}

\begin{Remark}
Support posets are important to find the depolarizations of a squarefree monomial ideal $I$, i.e. all ideals $J$ such that $J^P$ is isomorphic to $I$, where $J^P$ is the polarization of $J$ cf. \cite{MPSW19}. In particular, all monomial ideals in the same polarity class, i.e. those having the same polarization, have the same support poset. We define the {\em support poset} of a general monomial ideal as the support poset of its polarization. 
\end{Remark}

Given $n$ subsets $C_i$ of $\{1,\dots,n\}$ with $i\in C_i$ for all $i$ and the set $\P$ of all of them, we form the poset $(\P,\prec)$ on the elements of $\P$  ordered by inclusion. For such $(\P,\prec)$ we can construct (in principle several) monomial ideals $I_\P$ such that $(\P,\leq)$ is the support poset of $I_\P$ using the following result

\begin{Proposition}[Proposition 3.1 in \cite{MPSW19}]\label{prop:supportIdeals}
Let $(\C=\{C_1,\dots,C_n\},\subseteq)$ be a poset such that 
$\{i\}\subseteq C_i\subseteq [n]$ for each $i$, and if $k\in C_i$ and $i\in C_j$ then $k\in C_j$ for all $i,j,k$. 
Let $R=\kb[x_1,\dots,x_n]$ and let $m_i=\prod_{j\in C_i}x_j$ for each $i$. 
For any $\sigma\subseteq [n]$ let $m_\sigma=\lcm(m_i \vert i\in\sigma)$, and for any collection $\Sigma$ of subsets of $[n]$, consider the monomial ideal $I_\Sigma=\langle m_\sigma \vert \sigma\in\Sigma\rangle$.
Then $(\C,\subseteq)$ is the support poset of $I_\Sigma$ if the following properties hold:
\begin{enumerate}
\item $\forall i\in[n]$ there is some $\sigma \in\Sigma$ such that $x_i\vert m_\sigma$.
\item If $\{\sigma:\ x_i|m_\sigma\}\subseteq \{\sigma:\ x_j|m_\sigma\}$, then $C_j\subseteq C_i$.
\end{enumerate}
\end{Proposition}

The support poset of any monomial ideal $I\subseteq R=\kb[x_1,\dots,x_n]$, together with a given ordering $<$ of the set of the variables induces a partial order $\prec$ in the set of variables in the following way: $x_i\prec x_j$ if $C_i\subset C_j$ or if $C_i=C_j$ and $x_i<x_j$. We call this poset the {\em $<$-support poset} of $I$ and denote it $\suppPos_<(I)$. Observe that if $C_i\neq C_j$ for all $i\neq j$ then any $<$-support poset of $I$ is equal to $\suppPos(I)$ for any ordering $<$ of the variables. The Hasse diagram of $\suppPos_<(I)$ is equivalent to the Hasse diagram of $\suppPos(I)$ where any node $C$ labelled with more than one index is substituted by a vertical line of nodes labelled by the elements of the label of $C$, ordered by $<$.

\begin{Example}[Example 3.2 in \cite{MPSW19}]\label{ex:support-posets}
Let us consider the following sets:

\begin{enumerate}
\item Let 
$C_1=\{1,2\},\,C_2=\{2\},\, C_3=\{3\}, C_4=\{4\}$ and $C_5=\{4,5\}$.

\noindent Let $\Sigma_1=\{\{1\},\{2,4\},\{3\},\{5\}\}$, $\Sigma_2=\{\{1\},\{2,3\},\{3,4\},\{5\}\}$ and $\Sigma_3=\{\{1,3\},\{3,5\},\{1,4\},\{2,5\}\}$. These three collections satisfy the conditions in Proposition \ref{prop:supportIdeals} and hence $(\C=\{C_1,\ldots,C_5\},\subseteq)$ is the support poset of the ideals $I_{\Sigma_1}=\langle x_1x_2,x_2x_4,x_3,x_4x_5\rangle$, $I_{\Sigma_2}=\langle x_1x_2,x_2x_3,x_3x_4,x_4x_5\rangle$ and $I_{\Sigma_3}=\langle x_1x_2x_3,x_3x_4x_5,x_1x_2x_4,x_2x_4x_5\rangle$. 
\item Let $\C$ be given by $C_1=\{1\},\, C_2=\{1,2\}$ and $C_3=\{1,2,3\}$, then there is no monomial ideal $I\subseteq R[x_1,x_2,x_3]$ such that $(\C,\subseteq)$ is the support poset of $I$. To see this, observe that $x_1x_2x_3$ must be one of the minimal generators of $I$, hence the only one, but $\C$ is not the support poset of $I=\langle x_1x_2x_3\rangle$.
\item Let $C_1=\{ 1,2,4\}$, $C_2=\{ 1,2,4\}$, $C_3=\{ 1,2,3,4\}$, $C_4=\{ 4\}$, $C_5=\{ 1,2,4,5,6\}$, $C_6=\{ 4,6\}$, $C_7=\{7\}$, $C_8=\{7,8\}$, $C_9=\{ 7,8,9\}$, $C_{10}=\{ 7,8,10\}$. Then for $\Sigma=\{\{3\},\{6,7\},\{5\},\{9\},\{10\}\}$, the ideal 
\[
I_\Sigma=\langle x_1x_2x_3x_4,x_4x_6x_7,x_1x_2x_4x_5x_6,x_7x_8x_9,x_7x_8x_{10}\rangle\subseteq \kb[x_1,\dots,x_{10}].
\] 
has $(\mathcal{C}=\{C_1,\dots,C_{10}\},\subseteq)$ as its support poset. 
\end{enumerate}
\end{Example}

\begin{Remark}
Observe that $x_i\in C_j$ and $x_k\in C_i$ imply that $x_k \in C_j$ for all $i,j,k$. We can use this fact to visualize support posets using their Hasse diagrams, where each node is labelled by those indices that are in that node and not in any of the nodes below it.

\begin{figure}[h]
\centering
\begin{subfigure}[b]{0.48\textwidth}
\centering
      \begin{tikzpicture}[scale=.8, transform shape, vertices/.style={text width= 1.5em,align=center,draw=black, fill=white, circle, inner sep=1pt}]
             \node [vertices] (0) at (0,0){2};
             \node [vertices] (2) at (1.5,0){3};
             \node [vertices] (3) at (3,0){4};
             \node [vertices] (1) at (0,1.5){1};
             \node [vertices] (4) at (3,1.5){5};
              \foreach \to/\from in {0/1, 3/4}
     \draw [-] (\to)--(\from);
     \end{tikzpicture}
     \caption{Support poset for Example \ref{ex:support-posets} (1)}
     \label{fig:Hasse-support-a}
 \end{subfigure} 
\begin{subfigure}[b]{0.48\textwidth}
\centering
 \begin{tikzpicture}[scale=.8, transform shape, vertices/.style={text width= 1.5em,align=center,draw=black, fill=white, circle, inner sep=1pt}]
             \node [vertices] (0) at (3,0){4};
             \node [vertices] (1) at (7.5,0){7};
             \node [vertices] (2) at (1.5,1.5){1,2};
             \node [vertices] (3) at (4.5,1.5){6};
             \node [vertices] (4) at (7.5,1.5){8};
             \node [vertices] (5) at (0,3){3};
             \node [vertices] (6) at (3,3){5};
             \node [vertices] (7) at (6,3){9};
             \node [vertices] (8) at (9,3){10};
              \foreach \to/\from in {0/2, 0/3, 1/4, 2/5, 2/6, 3/6, 4/7,4/8}
     \draw [-] (\to)--(\from);		
     \end{tikzpicture}
     \caption{Support poset for Example \ref{ex:support-posets} (3)}
 \end{subfigure}
 \end{figure}
 
 \end{Remark}

Figure \ref{fig:hasse} shows the Hasse diagram of $\suppPos_\leq(I)$ in Example  \ref{ex:support-posets} (3) for any order $\leq$ in which $x_1<x_2$.

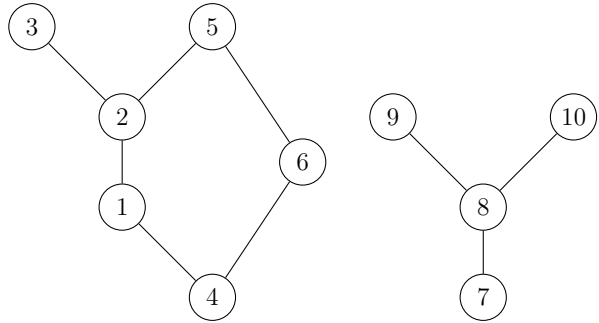
\begin{figure}[h]
\centering
 \begin{tikzpicture}[scale=.8, transform shape, vertices/.style={text width= 1.5em,align=center,draw=black, fill=white, circle, inner sep=1pt}]
             \node [vertices] (0) at (3,0){4};
             \node [vertices] (1) at (7.5,0){7};
             \node [vertices] (2) at (1.5,1.5){1};
             \node [vertices] (2b) at (1.5,3){2};
             \node [vertices] (3) at (4.5,2.25){6};
             \node [vertices] (4) at (7.5,1.5){8};
             \node [vertices] (5) at (0,4.5){3};
             \node [vertices] (6) at (3,4.5){5};
             \node [vertices] (7) at (6,3){9};
             \node [vertices] (8) at (9,3){10};
              \foreach \to/\from in {0/2, 0/3, 1/4, 2/2b, 2b/5, 2b/6, 3/6, 4/7,4/8}
     \draw [-] (\to)--(\from);		
     \end{tikzpicture}
     \caption{$\leq$-Support poset for Example \ref{ex:support-posets} (3) for any order such that $x_1<x_2$}
     \label{fig:hasse}
 \end{figure}

\subsection{Mayer-Vietoris trees}
Mayer-Vietoris trees were introduced in \cite{S09} as a tool to combinatorially obtain the support of mapping cone resolutions \cite{CE95,HT02}. Let $I\subseteq S=\kb[x_1,\dots,x_n]$ be a monomial ideal and $G(I)=\{g_1,\dots,g_r\}$ the unique minimal monomial generating set of $I$. Consider any ordering of the elements in $G(I)$ and let $I_i=\langle g_1\dots,g_i\rangle$ the subideal generated by the first $i$ generators of $I$. For each $i$ we have the following exact sequence
\begin{equation}\label{eq:ses}
0\longrightarrow I_{i-1}\cap \langle g_i\rangle\stackrel{j}{\longrightarrow}I_{i-1}\oplus\langle g_i\rangle\stackrel{ }{\longrightarrow}I_i\longrightarrow 0.
\end{equation}
Assume that free resolutions $\FF'_i$ and $\widetilde{\FF}_i$ are known for $I'_i=I_{i-1}$ and $\widetilde{I_i}=I_{i-1}\cap\langle g_i\rangle$ respectively. Then, a (not necessarily minimal) resolution $\FF_i$ of $I_i$ is obtained as the mapping cone of the chain complex morphism $\varphi: \widetilde{\FF_i}\longrightarrow \FF'_i$ that lifts the inclusion $j$.

Using recursively the sequence (\ref{eq:ses}) on $i$ we can compute a free resolution of $I$ that is called an {\em iterated mapping cone resolution}. The ideals involved in this process can be displayed as a binary tree. The root of this tree is $I$ and every node $J=\langle f_1\dots,f_r\rangle$ has $J'=\langle f_1,\dots,f_{r-1}\rangle$ as right child and $\widetilde{J}=J'\cap\langle f_r\rangle$ as left child. This is called a {\em Mayer-Vietoris tree} of $I$, cf. \cite{S09}.

Each node in a Mayer-Vietoris tree is assigned a position and a dimension. The root has position $1$ and dimension $0$ and the right and left children of a node with position $p$ and dimension $d$ are given positions $2p+1$ and $2p$ respectively, and dimensions $d$ and $d+1$ respectively. We say that a node is {\em relevant} if it is the root of the tree or if its position is even. The monomials of the relevant nodes of dimension $d$ in a Mayer-Vietoris tree are then the multidegrees of the generators of the $d$-th module of the iterated mapping cone resolution $\FF$ of $I$ described by the tree. Let $\MVT(I)_{d,\mu}$ be the set of generators of multidegree $\mu$ in the relevant nodes of dimension $d$ of a Mayer-Vitoris tree of $I$, and let $\MVT(I)'_{d,\mu}$ be the set of elements of $\MVT(I)_{d,\mu}$ that appear only once in relevant nodes of the tree. Since the minimal free resolution of $I$ is a subresolution of $\FF$ we have that for any Mayer-Vietoris tree the following result holds.

\begin{Proposition}\label{prop:MVTbounds}
For any Mayer-Vietoris tree of $I$
\[
\#\MVT(I)'_{d,\mu}\leq\beta_{d,\mu}(I)\leq\#\MVT(I)_{d,\mu}
\]

\end{Proposition}

The generators of the relevant nodes of $\MVT(I)$ provide upper and lower bounds for the Betti numbers of the ideal without actually computing the resolution. These bounds can be improved using several criteria and are sharp in several families of ideals, see \cite{S09} for details. A simple useful criterion is the following:
\begin{Proposition}
Let $\mu$ be a multidegree such that there are generators of multidegree $\mu$ in relevant nodes of $\MVT(I)$ of dimensions $d_1\dots d_k$ such that no two of them are consecutive, then 
\[
\beta_{d_i,\mu}(I)=\#\MVT(I)_{d_i,\mu}\,\forall i=1,\dots, k
\]
\end{Proposition}

\section{Ideals with a given support poset}\label{sec:givenPoset}
In this section we give several examples of posets that are realized as support poset of some monomial ideal and we explicitly find those ideals. First, we consider two families of posets,  and for every poset $\P$ in these families we describe an ideal $I_\P$ for which $\P$ is its support poset.  We compute the Betti numbers of these ideals. In the second part of the section we focus on trees and for any tree $\P$ we describe its leaf ideal, a monomial ideal whose support poset is $\P$ .

\subsection{Lines and diamonds}
As an application of support posets the authors give in \cite{MPSW19} the following result describing some ideals having a given poset as a support poset.

 \begin{Proposition}[Proposition 3.13 in \cite{MPSW19} ]\label{prop:npaths}
Let $n, m_1,\dots, m_n$ be some positive integers with $1\leq m_i\leq n$ for all $i$ and let $m=\sum_{i}m_i$. Consider a poset $(\P,\subseteq)$ on subsets of $\{1,\dots,m\}$ formed by $n$ disjoint paths each of length $m_i$. Then there is a squarefree monomial ideal $I$ whose support poset is $\P$ except if $n=2$ and $m_1\neq m_2$. Moreover, if $m_i>1$ for all $i$, then there is a zero-dimensional monomial ideal copolar to $I$.
\end{Proposition}

In the same spirit, we propose the following result.

\begin{Proposition}\label{prop:lines}
Let $n$ and $m$ be two positive integers and let $(\P,\subseteq)$ be a poset of subsets of the set $[nm]=\{1,\dots,nm\}$ formed by $n$ disjoint lines each of length $m$. Then there is at least one squarefree monomial ideal $I_{n,m}$ such that $\P$ is its support poset and there is a zero-dimensional monomial ideal $J_{n,m}$ copolar to $I_{n,m}$.

In particular, the ideal $J_{n,m}\subseteq\kb[y_1,\dots,y_{n}]$ given by
\[
J_{n,m}=\langle y_1^m,\dots,y_n^m, y_1^{m-1}y_2,\dots,y_1y_2^{m-1}, \dots, y_1^{m-1}y_n,\dots,y_1y_n^{m-1}\rangle.
\]
is a zero dimensional ideal having $\P$ as its support poset.
\end{Proposition}

\begin{proof}
We describe the construction of $I$ stepwise as $n$ increases. The base case is $n=2$. Let $\P=A_1\sqcup A_2$ where $A_1=\{1,\dots,m\}$ and $A_2=\{m+1,\dots,2m\}$. The ideal $I_{2,m}\subset \kb[x_1,\dots,x_{2m}]$ generated by the monomials
\[
x_1\cdots x_m, x_{m+1}\cdots x_{2m}, x_1\cdots x_{m-1}x_{m+1}, x_1\cdots x_{m-2}x_{m+1}x_{m+2},\dots, x_1x_{m+1}\cdots x_{2m-1}
\]
satisfies that $\suppPos(I_{2,m})=\P$.

To see this, let us consider first the indices in $A_1$. Observe that $x_m$ appears only in the first generator, $x_1\dots x_m$, hence $C_m=\{1,\dots,m\}$. If $1\leq j<m$ then every generator that contains $x_j$ also contains $x_1\dots,x_{j-1}$ and if $j<k\leq m$ then there is at least one generator which contains $x_j$ but not $x_k$, for instance $x_1\cdots x_jx_{m+1}\cdots x_{2m-j}$. Finally, if $k\geq m+1$ we have that $x_k$ is not present in $x_1\cdots x_m$ in which $x_j$ is, hence $C_j=\{1,\dots,j\}$. By simmetry, the same applies to the generators in $A_2$.

Considering in $\P$ the chain partition given by the $A_i$'s we have that the corresponding depolarization is $J_{2,m}\subset\kb[y_1,y_2]$ given by
\[
J_{2,m}=\langle y_1^m,y_2^m, y_1^{m-1}y_2,\dots,y_1y_2^{m-1}\rangle
\]
which is zero-dimensional and $J^P_{2,m}=I_{2,m}$.
Let now $n=3$. Then $I_{3,m}\subseteq\kb[x_1,\dots x_{3m}]$ is given by the same set of generators of $I_{2,m}$ plus the following ones
\[
\{ x_{2m+1}\cdots x_{3m}, x_1\cdots x_{m-1}x_{2m+1}, x_1\cdots x_{m-2}x_{2m+1}x_{2m+2},\dots, x_1x_{m+1}\cdots x_{3m-1}\}.
\]

Using the same argument as for $n=2$ we have that $\suppPos(I_{3,m})=A_1\sqcup A_2\sqcup A_3$.
The ideal $J_{3,m}\subseteq[y_1,y_2,y_3]$ is given by
\[
J_3=\langle y_1^m,y_2^m,y_3^m, y_1^{m-1}y_2,\dots,y_1y_2^{m-1}, y_1^{m-1}y_3,\dots,y_1y_3^{m-1}\rangle.
\]
Now, proceeding in the same way adding at each step the new generators 
\[
x_{(n-1)m+1}\cdots x_{nm}, x_1\cdots x_{m-1}x_{(n-1)m+1}, \dots, x_1x_{(n-1)m+1}\cdots x_{nm-1}
\]
we obtain the ideal $I_{n,m}$ whose support poset is formed by a disjoint set of $n$ paths of size $m$.
 
The ideal $J_{n,m}\subseteq\kb[y_1,\dots,y_{n}]$ is given by
\[
J_{n,m}=\langle y_1^m,\dots,y_n^m, y_1^{m-1}y_2,\dots,y_1y_2^{m-1}, \dots, y_1^{m-1}y_n,\dots,y_1y_n^{m-1}\rangle.
\]
Observe that $J^P_{n,m}=I_{n,m}$ and $J_{n,m}$ is zero-dimensional for all $n$, since it contains a pure power of each of the variables.
\end{proof}

Observe that the ideal constructed in Proposition \ref{prop:lines} can be obtained by taking the following collection $\sigma$ in $(\P,\subseteq)$
\[
\Sigma=\left( \bigcup_{i=1}^{n}\{im\}\right)\bigcup\left( \bigcup_{i=1}^{n-1}\bigcup_{j=1}^{m-1}\{m-j,im+j\}\right),
\]
which satisfies Proposition \ref{prop:supportIdeals}. The ideals $J_{n,m}$ are generated in degree $m$ and their Betti numbers are computed by the following result.

\begin{Proposition}\label{prop:Betti-lines}
The Betti numbers of the ideal 
\[
J_{n,m}=\langle y_1^m,\dots,y_n^m,y_1^{m-1}y_2,\dots,y_1y_2^{m-1},\dots,y_1^{m-1}y_n,\dots,y_1y_n^{m-1}\rangle
\]
are given by

\[
\beta_0(J_{n,m})=n+(n-1)(m-1)
\]
\[
\beta_i(J_{n,m})={{n-1}\choose {i}}+\sum_{j=2}^{n}(m-1){{1+n-j}\choose i}+{{n-1}\choose{i+1}}\;\forall\, 1\leq i\leq{n-1}
\]
In particular, $\projdim(J_{n,m})=n-1$ and $\reg(J_{n,m})=(n-1)(m-1)$. The minimal free resolution of $J_{n,m}$ can be obtained as an iterated mapping cone.
\end{Proposition}

\begin{proof}
We shall use Mayer-Vietoris trees. First we sort the generators of $J_{n,m}$ in the following way:
\[
y_1^m,\dots,y_n^m,
\]
\[
y_1y_2^{m-1},\dots,y_1y_n^{m-1},
\]
\[
...
\]
\[
y_1^{m-1}y_2,\dots,y_1^{m-1}y_n.
\]
Now, we use them in turn to construct a Mayer-Vietoris tree of $J_{n,m}$ i.e. an iterated mapping cone resolution. Let us denote this resolution by $\FF$ and let $\gamma_i(J_{n,m})$ denote the rank of the $i$'th module of $\FF$. We proceed row by row with the pivots.

The first pivot, $y_1^m$ produces the ideal $\langle y_1^m\rangle \cap \langle y_2^m,\dots,y_1^{m-1}y_n\rangle$, minimally generated by $\langle y_1^m y_2,\dots,y_1^m y_n\rangle$. Observe that the Taylor complex of this ideal is its minimal free resolution, and hence the contribution of this ideal to $\gamma_i(J_{n,m})$ is ${{n-1}\choose {i}}$ for $1\leq i\leq n-1$. Each of the next $n-1$ pivots in the first row, namely $y_2^m,\dots y_n^m$ produce the ideals $\langle y_1y_j^m,y_j^my_{j+1}^m, \dots, y_j^my_n^m \rangle$, $2\leq j\leq n$. Each of these ideals is again minimally resolved by its Taylor complex and hence their contribution to $\gamma_i(J_{n,m})$ is ${{1+n-j}\choose {i}}$ for $1\leq i\leq n-1$ and $2\leq j\leq n$.

For the next $m-2$ rows of pivots, from $y_1y_2^{m-1},\dots, y_1y_n^{m-1}$ to $y_1^{m-2}y_2^2,\dots, y_1^{m-2}y_n^2$ we have that each pivot $y_1^{m-k}y_j^k$ with $j=2,\dots,n$ and $k=2,\dots,m-1$ produces the ideal
$$\langle y_1^{m-k+1}y_j^k,y_1^{m-k}y_j^k y_{j+1}^k,\dots,y_1^{m-k}y_j^k y_n^k\rangle.$$
All these ideals are again minimally resolved by their Taylor complexes and hence their contribution to $\gamma_i(J_{n,m})$ is ${{1+n-j}\choose {i}}$ for $1\leq i\leq n-1$, and this is for $2\leq j \leq n$ and $2\leq k\leq m-1$ hence the contribution of these rows to $\gamma_i(J_{n,m})$ is $\sum_{j=2}^n (m-2){{1+n-j}\choose {i}}$ for $1\leq i\leq n-1$.

Finally, the last row forms a monomial ideal whose Taylor complex is its minimal resolution and is generated by $n-1$ monomials , hence its contribution to $\gamma_i(J_{n,m})$ is ${{n-1}\choose {i+1}}$ for $1\leq i\leq n-1$.

Putting all these contributions together we have that
\[
\gamma_i(J_{n,m})= {{n-1}\choose {i}}+\sum_{j=2}^{n}(m-1){{1+n-j}\choose i}+{{n-1}\choose{i+1}}\;\forall\, 1\leq i\leq{n-1}.
\]
Now it is easy to observe, by the sorting of our pivots, that the Mayer-Vitoris tree that we have built has no repeated generators, i.e. the generators of the modules in $\FF$ all have different multidegrees, hence $\FF$ is minimal. And we obtain that $\projdim(J_{n,m})=n-1$.

Finally, observe that the ideal produced by pivot $y_2^m$ is $J=\langle y_1y_2^m,y_2^my_3^m,\dots y_2^my_n^m\rangle$ $\reg(J)=(n-1)(m-1)+1$. By easy inspection of the degrees of the rest of ideals involved, we can see that $\reg(J_{n,m})=\reg(J)-1$ .
\end{proof}

\begin{Remark}
By keeping track of the (multi-)degrees of the generators of the ideals in the Mayer-Vietoris tree built in Proposition \ref{prop:Betti-lines} we obtain the (multi-)graded Betti numbers of $J_{n,m}$.
\end{Remark}

\begin{Proposition}\label{prop:diamonds}
Let $m$ be a positive integer, let $(\P,\subseteq)$ be a poset of subsets of the set $[4m]$ formed by $m>1$ disjoint diamonds $D_1,\dots,D_m$, $D_i=\{a_{i1},\dots,a_{i4}\}$ with $a_{i1}<a_{i2},\, a_{i1}<a_{i3},\, a_{i2}<a_{i4},\,a_{i3}<a_{i4}$. Then there is at least one squarefree monomial ideal $I_m$ such that $\P$ is its support poset.\end{Proposition}

\begin{proof}
Consider the following two sets of monomials:
\begin{itemize}
\item[] $A=\{ x_{11}x_{12}x_{13}x_{14},\dots,x_{m1}x_{m2}x_{m3}x_{m4}\}$
\item[]  $B=\{x_{11}x_{12}x_{21}x_{23},\dots,x_{(m-1)1}x_{(m-1)2}x_{m1}x_{m3},x_{m1}x_{m2}x_{11}x_{13}\}$
\end{itemize}
Let $\bar{I}_m=\langle A\cup B\rangle\subseteq\kb[x_{11},\dots,x_{14},\dots,x_{m1},\dots, x_{m4}]$, then $\P$ is in fact the support poset of $\bar{I}_m$. Just observe that for every $i$ we have that $x_{i4}$ is only present in the monomial $x_{i1}x_{i2}x_{i3}x_{i4}$ hence $C_{i4}=\{i1,i2,i3,i4\}$, $x_{i3}$ is present in the monomials $x_{i1}x_{i2}x_{i3}x_{i4}$ and $x_{i-1,1}x_{i-1,2}x_{i1}x_{i3}$ hence $C_{i3}=\{i1,i3\}$. The variable $x_{i2}$ is present in the monomials $x_{i1}x_{i2}x_{i3}x_{i4}$ and $x_{i1}x_{i2}x_{i+1,1}x_{i+1,3}$ hence $C_{i2}=\{i1,i2\}$. Finally $x_{i1}$ is present in the monomials $x_{i1}x_{i2}x_{i3}x_{i4}$,  $x_{i1}x_{i2}x_{(i+1)1}x_{(i+1)3}$ and $x_{(i-1)1}x_{(i-1)2}x_{i1}x_{i3}$ hence $C_{i1}=\{i1\}$ \footnote{If $i=1$ then take $m$ instead of $i-1$, and if $i=m$ take $1$ instead of $i+1$.}.
\end{proof}

One possible partition of $(\P,\subseteq)$ is to consider, for each $i$ the paths $\{a_{i1},a_{i2},a_{i4}\}$ and $\{a_{i3}\}$, the resulting deporalization is an ideal for which we can explicitly compute the Betti numbers, hence obtaining the Betti numbers of all the ideals in its polarity class.

\begin{Proposition}\label{prop:diamondsBetti}
Let $I_m\subseteq\kb[x_{11},x_{12},\dots,x_{m1},x_{m2}]$ the ideal given by
\[
I_m=\langle x_{11}^3x_{12},\dots,x_{m1}^3x_{m2},x_{11}^2x_{2,1}x_{2,2},\dots,x_{(m-1)1}^2x_{m1}x_{m2},x_{m1}^2x_{11}x_{12}\rangle.
\]
The Betti numbers of $I_m$ are given by 
\[
\beta_0(I_m)=2m
\]
\[
\beta_i(I_m)=2K^m_{m-3,i-1}+K^m_{m-2,i}
\] 
where the numbers $K^m_{a,b}$ are given by the recurrence relation
\[
K^m_{a,b}=K^m_{a-2,b-1}+K^m_{a-1,b}
\]
with base cases
\[K^m_{a,0}=m+a,\, K^m_{0,i}={m\choose{i+1}},\, K^m_{1,i}={m\choose i}+{m\choose{i+1}}.
\]
\end{Proposition}

\begin{proof}
We divide the generators of $I_m$ in two groups $A$ and $B$. Group $A$ consists on the following $m$ generators: $x_{11}^3x_{12},\dots,x_{m1}^3x_{m2}$. Group $B$ consists on the following $m$ generators:
$$x_{11}^2x_{21}x_{22},\dots,x_{(m-1)1}^2x_{m1}x_{m2},x_{m1}^2x_{11}x_{12}.$$
Since $I_m$ has $m$ generators in each of the groups we say that it is of the form $\langle m\vert m\rangle$.

To build the Mayer-Vietoris tree of $I_m$ we will first use the generators of group $A$ in the given order. When using the first generator, the ideal produced is given by the monomial $x_{11}^3x_{12}$ multiplied by each of the following $m-3$ generators from group $A$: $x_{3,1}^3x_{3,2}, \dots, x_{m-1,1}^3x_{m-1,2}$, and the following $m$ monomials, one for each generator of group $B$: $x_{2,1}x_{2,2},x_{2,1}^2x_{3,1}x_{3,2},\dots,x_{m-1,1}^2x_{m1}x_{m2},x_{m1}^2$. I.e. the obtained ideal $\widetilde{I}_m$ is of the form $\langle m-3\vert m\rangle$. The ideal $I'_m=\langle x_{21}^3x_{22},\dots,x_{m1}^2x_{11}x_{12}\rangle$ is of the form $\langle m-1\vert m\rangle$.

We continue the construction of the Mayer-Vietoris tree by using as pivots the monomials in group $A$ in their given order. If we take a pivot from an ideal of the form $\langle a\vert m\rangle$ then its left child is of the form $\langle a-2\vert m\rangle$ (or $\langle 0\vert m\rangle$ if $a\leq 2$) and the right child is of the form $\langle a-1\vert m\rangle$. Each time, when using the pivot $x_{i1}^3x_{(i+1)1}x_{(i+1)2}$ we delete generators $x_{(i+1)1}^3x_{(i+1)2}$ from group $A$ and transform the generators $x_{i1}^2x_{(i+1)1}x_{(i+1)2}$ and $x_{i-1}^2x_{i1}x_{i2}$ into $x_{i1}^3x_{i2}x_{(i+1)1}x_{(i+1)2}$ and $x_{i1}^3x_{i2}x_{i-1}^2$ respectively (observe that when we use $x_{m1}^3x_{m2}$ we transform $x_{m1}x_{11}x_{12}$ into $x_{m1}^3x_{m2}x_{11}x_{12}$).

We continue this procedure until we reach an ideal of the form $\langle 0\vert m\rangle$. Ideals of this form are minimally resolved by their Taylor complex, no matter how we choose pivots, since they consist of the list of generators $x_{11}^2x_{21}x_{22},\dots, x_{m1}^2x_{11}x_{12}$ where some of them have been substituted by their corresponding $x_{i1}^3x_{i2}x_{(i+1)1}x_{(i+1)2}$ or by $x_{(i-1)1}^3x_{(i-1)2}x_{i1}^2$. No $\lcm$ of any set of $i$ of these monomials is divisible by the $\lcm$ of any other set of $i$ of them.

The ideals of the form $\langle a\vert m\rangle$ for $a>1$ which are in an even position of dimension $i$ of the tree contribute with $a+m$ generators to $\beta_i(I_m)$. The nodes of the form $\langle 0\vert m\rangle$ in an even position of dimension $i$ of the tree contribute with ${m \choose {j-i+1}}$ generators to $\beta_j(I_m)$ for $j\geq i$. Finally, the nodes $\langle 0\vert m \rangle$ in an odd position of dimension $i$ of the tree contribute to $\beta_j(I_m)$ with ${ m\choose {j-i+1}}$ generators.

Now we add up all the contributions. The number $K^m_{a,i}$ for $a>0$ represents the contribution of a node of the form $\langle a\vert m\rangle$ to $\beta_i(I_m)$. From the above considerations we have that $K^m_{a,i}=K^m_{a-1,i}+K^m_{a-2,i-1}$ and the base cases of this recursion are $K^m_{a,0}=m+a$ for $a>0$, $K^m_{0,i}={m\choose {i+1}}$ and $K^m_{1,i}={m\choose i}+{m\choose{i+1}}$ for $i>0$. Finally, from the first step in the construction of the tree, we have that $\beta_i(I_m)=K^m_{m-3,i-1}+K^m_{m-1,i}=2K^m_{m-3,i-1}+K^m_{m-2,i}$.
\end{proof}

We can use of the following binomial identity \cite{B18}, to describe the Betti numbers of $I_m$ in a direct non-recursive way.

\begin{Proposition}
\[
K^m_{a,b}={{a+1}\choose 0}{m\choose{b+1}}+{{a}\choose 1}{m\choose{b}}+{{a-1}\choose 2}{m\choose{b-1}}+\cdots
\]
\end{Proposition}

By direct inspection of the Mayer-Vietoris tree constructed in Proposition \ref{prop:diamondsBetti} we have that 

\begin{Corollary}
For every ideal $J_m$ in the polarity class of $I_m$ we have $\reg(J_m)=2m$, $\projdim(J_m)=\lfloor\frac{m}{2}\rfloor+m-1$ and its minimal free resolution is given as an iterated cone resolution.
\end{Corollary}

\subsection{Leaf ideals of trees and forests}
A more general class of posets are trees and forests. For them we can identify supported monomial ideals for which we can compute the main invariants.
\begin{Proposition}\label{prop:leafIdeal}
Let $\P$ be a tree with nodes $\{1,\dots, n\}$ and let  $\{l_1,\dots, l_k\}\subseteq \{1,\dots,n\}$ be the set of leaves of the tree. There exists a squarefree monomial ideal $I_{L}(\P)\subseteq \kb[x_1,\dots,x_n]$ with $k$ generators such that  $\P$ is it support poset. The Taylor resolution of  $I_{L}(\P)$ minimally resolves it and therefore $\beta_i(I_{L}(\P))={k\choose i+1}$ for all $i\geq 0$.
\end{Proposition}

\begin{proof}
Consider the ideal $I_L(\P)=\langle m_{l_1},\dots,m_{l_k}\rangle$ where $m_{l_i}=\prod_{i<l_i}x_i$; here $i<j$ means that $i$ is an ancestor of $j$. We have that $\P$ is the support poset of $I_L(\P)$. To see this, observe that given a variable $x_i$, the set $C_i$ of variables that appear in every generator in which $x_i$ appears is formed by the variables $x_j$ such that $j<i$ in $\P$.

To see that the Taylor complex of $I_L(\P)$ minimally resolves it and therefore $\beta_i(I_L(\P))={k\choose i+1}$ for all $i$, consider the following process:

First, for each node $a$ such that it is the unique child of node $b$ we identify both in a new node $\overline{a}$ and in the corresponding ideal we substitute $x_ax_b$ by $x_{\overline{a}}$. We proceed in the same way until we obtain a reduced tree $\P'$ such that each node is either a leaf or has more than one child. Observe that $I_L(\P)\simeq I_L(\P')$ and hence $\beta_i(I_L(\P))=\beta_i(I_L(\P'))$ for all $i$.

Now, take the root $\overline{a}$ of the reduced tree $\P'$ whose children are $\overline{b}_1,\dots,\overline{b}_{m}$. Delete $\overline{a}$ and we are left with a set of $m$ disjoint trees $\P'_1,\dots,\P'_m$ whose nodes are either leaves or have more than one children. Observe that $I_L(\P')= x_{\overline{a}}\cdot \sum_{i=1}^m I_L(\P'_i)$, where multiplication by $x_{\overline{a}}$ means that we multiply each generator of each of the ideals $I_L(\P'_i )$ by $x_{\overline{a}}$ and the ideals $I_L(\P'_i)$ are supported on mutually disjoint sets of variables, hence the sum is direct and $\beta_i(I_L(\P'))=\sum_{j=1}^m \beta_i(I_L(\P'_j))$.

By repeated us of this process we obtain that $I_L(\P)$ has the same total Betti numbers than a prime monomial ideal generated by  one variable for each of its leaves, and hence the result.
\end{proof}

\begin{Example}\label{ex:leafIdeal}
Consider the tree depicted in Figure \ref{fig:leafIdeal} together with its reducing process as described in Proposition \ref{prop:leafIdeal}. We have that the leaf ideal of $\P$ is given by
$$I_L(\P)=\langle x_1x_2x_3x_4,x_1x_2x_3x_5x_6,x_1x_2x_3x_7x_8x_9,x_1x_2x_3x_7x_8x_{10}x_{11},x_1x_2x_3x_7x_8x_{10}x_{12}\rangle,$$
which has the same total Betti numbers than the prime ideal generated by the variables corresponding to the leaves. i.e. $\langle x_4,x_6,x_9,x_{11},x_{12}\rangle$.
 
\begin{figure}[h]
\centering
\begin{subfigure}[b]{0.15\textwidth}
\centering
      \begin{tikzpicture}[scale=.5, transform shape, vertices/.style={text width= 1.5em,align=center,draw=black, fill=white, circle, inner sep=1pt}]
             \node [vertices] (0) at (3,0){1};
             \node [vertices] (1) at (3,1.5){2};
             \node [vertices] (2) at (3,3){3};
             \node [vertices] (3) at (1.5,4.5){4};
             \node [vertices] (4) at (3,4.5){7};
             \node [vertices] (5) at (4.5,4.5){5};
             \node [vertices] (6) at (3,6){8};
             \node [vertices] (7) at (4.5,6){6};
             \node [vertices] (8) at (2,7.5){9};
             \node [vertices] (9) at (4,7.5){10};
             \node [vertices] (10) at (3,9){11};
             \node [vertices] (11) at (5,9){12};

              \foreach \to/\from in {0/1, 1/2, 2/3, 2/4, 2/5, 5/7, 4/6, 6/8, 6/9,9/10,9/11}
     \draw [-] (\to)--(\from);		
     \end{tikzpicture}
 \end{subfigure} 
\begin{subfigure}[b]{0.15\textwidth}
\centering
\begin{tikzpicture}[scale=.5, transform shape, vertices/.style={text width= 1.5em,align=center,draw=black, fill=white, circle, inner sep=1pt}]
             \node [vertices] (2) at (3,3){$\overline{3}$};
             \node [vertices] (3) at (1.5,4.5){4};
             \node [vertices] (5) at (4.5,4.5){$\overline{6}$};
             \node [vertices] (6) at (3,6){$\overline{8}$};
             \node [vertices] (8) at (2,7.5){9};
             \node [vertices] (9) at (4,7.5){10};
             \node [vertices] (10) at (3,9){11};
             \node [vertices] (11) at (5,9){12};

              \foreach \to/\from in {2/3, 2/6, 2/5, 6/8, 6/9,9/10,9/11}
     \draw [-] (\to)--(\from);		
     \end{tikzpicture}
 \end{subfigure}
 \begin{subfigure}[b]{0.15\textwidth}
\centering
\begin{tikzpicture}[scale=.5, transform shape, vertices/.style={text width= 1.5em,align=center,draw=black, fill=white, circle, inner sep=1pt}]
             \node [vertices] (3) at (1,4.5){4};
             \node [vertices] (5) at (5,4.5){$\overline{6}$};
             \node [vertices] (6) at (3,4.5){$\overline{8}$};
             \node [vertices] (8) at (2,6){9};
             \node [vertices] (9) at (4,6){10};
             \node [vertices] (10) at (3,7.5){11};
             \node [vertices] (11) at (5,7.5){12};

              \foreach \to/\from in {6/8, 6/9,9/10,9/11}
     \draw [-] (\to)--(\from);		
     \end{tikzpicture}
      \end{subfigure}     
       \begin{subfigure}[b]{0.20\textwidth}
\centering
\begin{tikzpicture}[scale=.5, transform shape, vertices/.style={text width= 1.5em,align=center,draw=black, fill=white, circle, inner sep=1pt}]
             \node [vertices] (3) at (1,4.5){4};
             \node [vertices] (5) at (6,4.5){$\overline{6}$};
             \node [vertices] (8) at (2,4.5){9};
             \node [vertices] (9) at (4,4.5){10};
             \node [vertices] (10) at (3,6){11};
             \node [vertices] (11) at (5,6){12};

              \foreach \to/\from in {9/10,9/11}
     \draw [-] (\to)--(\from);		
     \end{tikzpicture}
      \end{subfigure}
             \begin{subfigure}[b]{0.20\textwidth}
\centering
\begin{tikzpicture}[scale=.5, transform shape, vertices/.style={text width= 1.5em,align=center,draw=black, fill=white, circle, inner sep=1pt}]
             \node [vertices] (3) at (1,4.5){4};
             \node [vertices] (5) at (5,4.5){$\overline{6}$};
             \node [vertices] (8) at (2,4.5){9};
             \node [vertices] (9) at (3,4.5){11};
             \node [vertices] (10) at (4,4.5){12};
     \end{tikzpicture}
      \end{subfigure}

  \caption{Reduction process of tree $\P$ in Example \ref{ex:leafIdeal}}
  \label{fig:leafIdeal}
 \end{figure}
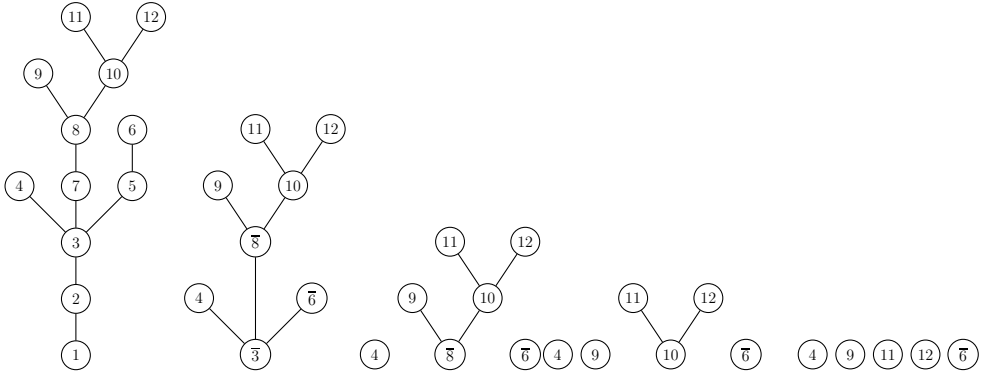
\end{Example}

We call $I_L(\P)$ the {\em leaf ideal} of $\P$. Generalizing Proposition \ref{prop:leafIdeal} we obtain the following result.
\begin{Theorem}\label{th:leafIdeal}
Let $\P$ be a forest whose trees $\P_1,\dots, \P_m$ have $n_i$ nodes and $l_i$ leaves each, for $i=1,\dots,m$. Then there is a squarefree monomial ideal $I_L(\P)\subseteq\kb[x_1,\dots,x_n]$, $n=\sum_{i=1}^m n_i$ whose support poset is $\P$. The ideal $I_L(\P)$ has $g=\sum_{i=1}^m l_i$ minimal monomial generators, its Taylor complex minimally resolves it, and $\beta_i(I_L(\P))=\sum_{j=1}^m{l_j\choose {i+1}}$ for all $i$.
\end{Theorem}

\section{Support posets of some families of ideals}\label{sec:givenIdeal}
We turn now to the second question that we address in this paper: given a certain class of monomial ideals, how can we describe their support posets? We will treat consecutive $k$-out-of-$n$ ideals (equivalently path ideals of line graphs cf. \cite{HV10}) and series-parallel ideals, i.e. path ideals of series-parallel systems cf. \cite{SW10}.
\subsection{Consecutive linear $k$-out-of-$n$ ideals}
A $k$-out-of-$n$ ideal $I_{k,n}\subseteq \kb[x_1,\dots,x_n]$ is an ideal generated by all possible products of $k$ variables. One can easily see that the support poset of such ideals is a collection of $n$ isolated points, hence it is the only ideal in its polarity class. Consecutive $k$-out-of-$n$ ideals are generated by the products of any $k$ consecutive variables, $J_{k,n}=\langle x_1\cdots x_k, x_2\cdots x_{k+1},\dots,x_{n-k+1}\cdots x_n\rangle$. These ideals are the path ideals of the line graph, and their main characteristics are well known \cite{HV10,SW11}. Here we describe their support poset.
\begin{Proposition}
Let $J_{k,n}=\langle x_1\cdots x_k, x_2\cdots x_{k+1},\dots,x_{n-k+1}\cdots x_n\rangle$ a consecutive $k$-out-of-$n$ ideal. Then its support poset $\P_{k,n}$ is given by
\begin{itemize}
\item[-]{If $k<n-k+1$:
 \[ 
 C_i=
 \begin{cases} 
      \{i,\dots,k\} & i=1,\dots,k \\
      \{i\} & i=k+1,\dots,n-k \\
      \{n-k+1,\dots,i\} & i=n-k+1,\dots,n  
   \end{cases}
\]
}
\item[-]{If $k\geq n-k+1$:
 \[
 C_i=
  \begin{cases} 
      \{i,\dots,k\} & i=1,\dots,n-k \\
      \{n-k+1,\dots,k\} & i=n-k+1,\dots, k\\
      \{n-k+1,\dots,i\} & i=k+1,\dots,n  
   \end{cases}
\]
}
\end{itemize}
\end{Proposition}

The form of $\P_{k,n}$ is given in Figure \ref{fig:supportPosetConsKNl}

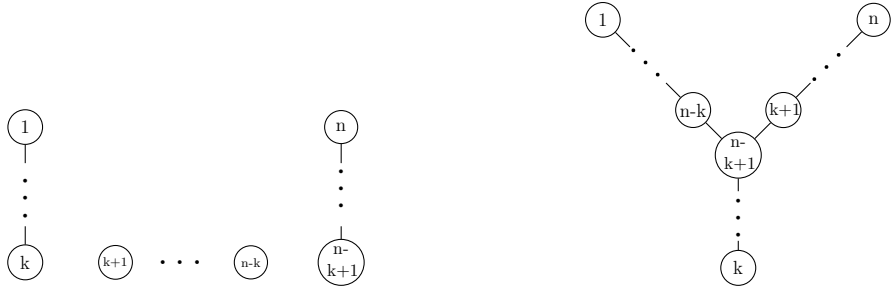
\begin{figure}[h]
\centering
\begin{subfigure}[b]{0.48\textwidth}
\centering
\begin{tikzpicture}[scale=.6, transform shape, vertices/.style={text width= 1.5em,align=center,draw=black, fill=white, circle, inner sep=1pt}]
             \node [vertices] (1) at (1,0){k};
             \node [vertices] (2) at (1,3){1};

             \node [vertices] (3) at (3,0){{\tiny k+1}};
             \node [vertices] (4) at (6,0){{\tiny n-k}};
             \node [vertices] (5) at (8,0){{\small n-k+1}};
             \node [vertices] (6) at (8,3){n};
        
      \path (1) -- (2) node [font=\Huge, midway, sloped] {$\dots$};
      \path (3) -- (4) node [font=\Huge, midway, sloped] {$\dots$};
      \path (5) -- (6) node [font=\Huge, midway, sloped] {$\dots$};
      
      \draw [-] (1)--(1,.8); \draw [-] (2)--(1,2.2);
      \draw [-] (5)--(8,.8); \draw [-] (6)--(8,2.2);
\end{tikzpicture}
\caption{Support poset of $J_{k,n}$ for $k<n-k+1$.}
 \end{subfigure} 
\begin{subfigure}[b]{0.48\textwidth}
\centering
\begin{tikzpicture}[scale=.6, transform shape, vertices/.style={text width= 1.5em,align=center,draw=black, fill=white, circle, inner sep=1pt}]
             \node [vertices] (1) at (1,3){1};
             \node [vertices] (2) at (3,1){n-k};
             \node [vertices] (3) at (4,0){{\small n-k+1}};
             \node [vertices] (4) at (4,-2.5){k};
             \node [vertices] (5) at (5,1){{\small k+1}};
             \node [vertices] (6) at (7,3){n};
        
      \path (1) -- (2) node [font=\Huge, midway, sloped] {$\dots$};
      \path (3) -- (4) node [font=\Huge, midway, sloped] {$\dots$};
      \path (5) -- (6) node [font=\Huge, midway, sloped] {$\dots$};
      
      \draw [-] (2)--(3); \draw [-] (1)--(1.6,2.4);\draw [-] (2)--(2.4,1.6);
      \draw [-] (3)--(5); \draw [-] (5)--(5.6,1.6);\draw [-] (6)--(6.4,2.4);
      \draw [-] (3)--(4,-0.8); \draw [-] (4)--(4,-1.9);

\end{tikzpicture}
\caption{Support poset of $J_{k,n}$ for $k\geq n-k+1$.}
 \end{subfigure} 

  \caption{Form of the support posets of consecutive $k$-out-of-$n$ ideals.}
  \label{fig:supportPosetConsKNl}
 \end{figure}
\begin{proof}
Observe that $x_1$ is only present in the generator $x_1\cdots x_k$ hence $C_1=\{1,\dots,k\}$. $x_2$ is present in generators  $x_1\cdots x_k$ and  $x_2\cdots x_{k+1}$ hence $C_2=C_1-\{1\}$, then $C_3=C_2-\{2\}$ and so on. By symmetry, $x_n$ is only present in generator $x_{n-k+1}\cdots x_n$ so $C_n=\{x_{n-k+1},\dots,x_n\}$ and $C_{n-1}=C_n-\{n\}$ etc.

If $k<n-k+1$ then we get in this way $C_1,\dots,C_k$ and $C_{n-k+1},\dots,C_n$. For all the $C_i$ with $k<i<n-k+1$ observe that $i$ appears in generators $x_{i-k+1}\cdots x_i$ to $x_i\dots x_{i+k-1}$ and these generators have only one variable in common, namely $x_i$, hence $C_i=\{i\}$.

If $k\geq n-k+1$ observe that $C_{n-k+1}=\cdots=C_{k}=\{n-k+1,\dots, k\}$ since the monomial $x_{n-k+1}\cdots x_k$ divides every generator of $J_{k,n}$, and for each $j<n-k+1$ the variable $x_j$ is only present in generators $x_i\cdots x_k$ for $i<j$ (and by symmetry, for every $j>k$ variable $x_j$ is only present in generators $x_i\cdots x_n$ for $i>j$).
\end{proof}

\begin{Remark}
Using the depolarization poset as described in \cite{MPSW19} we see that if $k\geq n-k+1$ there is a monomial ideal $J'_{k,n}$ copolar to $J_{k,n}$ in only two variables, namely 
\[
J'_{k,n}=\langle a^k,a^{k-1}b,\dots, a^{2k-n}b^{n-k}\rangle,
\]
which is isomorphic to the zero-dimensional ideal in two variables 
\[
J''_{k,n}=\langle a^{n-k},a^{n-k-1}b,\dots,b^{n-k}\rangle.
\]

If $k<n-k+1$ then we have that there is an ideal in $2+n-2k$ variables copolar to $J_{k,n}$, namely 
\[
J'_{k.n}=\langle a^k,a^{k-1}b_1,\dots,a^{3k-n}b_1\cdots b_{n-2k},a^{3k-n-1}b_1\cdots b_{n-2k}c,\dots
\]
\[
\dots,ab_1\cdots b_{n-2k}c^{3k-n-1}, b_1b_{n-2k}c^{3k-n},\dots,b_{n-2k}c^{k-1},c^k\rangle.
\]
These reductions in the number of variables improve drastically the computation times of the Betti numbers and other invariants for this kind of ideals.
\end{Remark}

\begin{Remark}\label{rm:differentIdeals}
The support poset $\P$ of a $J_{k,n}$ ideal is always a tree or forest. We could then use Proposition \ref{prop:leafIdeal} to construct the leaf ideal $I_L(\P)$ of $\P$. Observe that $I_L(\P)\neq J_{k,n}$. This is an example that a given poset $\P$ may be the support poset of different squarefree ideals.
\end{Remark}

\subsection{Series-parallel ideals}
Series-parallel ideals are defined as the cut ideals of series-parallel networks, a prominent class of coherent systems, cf. \cite{SW10}. We can define these ideals in the following way
\begin{Definition}
The ideal $I=\langle x_1\rangle\subseteq \kb[x_1]$ is called a {\em basic series-parallel} ideal. If $I_1\subseteq \kb[x_1,\dots,x_n]$ and $I_2\subseteq \kb[x_{n+1},\dots,x_{n+m}]$ are series-parallel ideals then $I'_1+I'_2$ and $I'_1\cap I'_2$ in $\kb[x_1,\dots,x_{n+m}]$ are series-parallel ideals, where $I'_1$ is the image of $I_1$ under the inclusion $\kb[x_1,\dots,x_n]\subseteq\kb[x_1,\dots,x_{n+m}]$ and $I'_2$ is the image of $I_2$ under the inclusion $\kb[x_{n+1},\dots,x_{n+m}]\subseteq\kb[x_1,\dots,x_{n+m}]$.
\end{Definition}

\begin{Theorem}\label{th:spIdeal}
The support poset of any series-parallel ideal is a forest.
\end{Theorem}
\begin{proof}
Let $I$ be a series-parallel ideal. We will give a constructive proof following a building process of $I$.

First, the support poset of a basic series-parallel ideal $I=\langle x_1\rangle$ has a single element $1$, which is a basic forest.

Now, we start constructing $I$ by joining basic series-parallel ideals, i.e. ideals of the form $\langle x_i \rangle$ one at a time. If we join  $\langle x_i \rangle$ and  $\langle x_j \rangle$ by union, the resulting ideal is  $\langle x_i,x_j \rangle$ whose poset is the disjoint union of two points. If we join them by intersection, we obtain  $\langle x_ix_j \rangle$ whose poset is a line with two points. Whenever we join a new basic series-parallel ideal  $\langle x_i \rangle$ we either join it by addition, in which case we have a new disjoint point in the support poset of the new ideal, or we add it by intersection, in which case we obtain the new poset by setting $i$ as its unique minimal element and joining the minimal element of each connected component of the previous support poset to $i$. Hence, whenever our series-parallel ideal is built by joining on new basic series-parallel ideal at a time its support poset is a tree plus zero or more disjoint points.

The next step is joining two of these ideals $I_1\subseteq\kb[x_1,\dots,x_n]$ and $I_2\subseteq\kb[x_{n+1},\dots,x_{n+m}]$ whose support posets we denote by $T_1$ and $T_2$, and the support poset of the resulting ideal by $T$. If $I=I_1+I_2$ then $T$ is the disjoint union of $T_1$ and $T_2$ since the two ideals have separate sets of variables.

If $I=I_1\cap I_2$ then the minimal monomial generating set of $I$ is given by all the products $\{m_im'_j\vert m_i \mbox { is a generator of } I_1,\, m'_j \mbox{ is a generator of } I_2\}$ and we can be in one of the following three cases:

\begin{enumerate}
\item[i)] $T_1$ and $T_2$ have more than one connected component each. In this case, $T$ is the disjoint union of $T_1$ and $T_2$. To see this, observe that there are no indices $i\in\{1,\dots,n\}$ and $j\in\{n+1,\dots,n+m\}$ such that $i<j$ or $j<i$. If this was not the case, assume we have $i<j$, then whenever $x_j$ appears in a generator of $I$ so does $x_i$. But we have that  for every generator $g$ of $I_1$ there is a generator in $I$ of the form $\mu x_{j}g$ with $\mu\in\kb[x_{n+1},\dots,x_{n+m}]$, hence $x_i$ is in every generator of $I_1$ and hence $T_1$ has one connected component which contradicts our asumption.
\item[ii)] Either $T_1$ or $T_2$ have one connected component. Let $T_1$ have a single connected component, then $T_1$ is a tree which has a set of elements $b=\{i_1,\dots, i_k\}\subseteq\{1,\dots,n\}$ such that $C_{i_1}=\dots=C_{i_k}$ and $j>i_a$ for every $i_a\in b$, $j\in\{1,\dots, n\}$ not in $b$, i.e. $b$ is the {\em trunk} of the tree $T_1$. Then $T$ is formed by the union of $T_1$ and $T_2$ plus a connection from $\max(b)$ to every minimal element in $T_2$. This is because in $I$ there are no new relations among the variables $\{x_1,\dots,x_n\}$ or among the variables in $\{x_{n+1},\dots,x_{n+m}\}$. Observe that all variables in $b$ appear in every generator of $I$.

\item[iii)] Both $T_1$ and $T_2$ have only one connected component each. Then let $b$ and $b'$ be the maximal elements of their respective trunks. The support poset of $T$ is the result of identifying both trunks.
\end{enumerate}
\end{proof}

\begin{Remark}
Observe that Theorem \ref{th:spIdeal} provides a criterion to detect ideals that cannot be obtained as a series-parallel ideal. For instance, the ideal $I$ in Example \ref{ex:support-posets} (3) cannot be obtained as a series-parallel ideal since its support poset is not a forest.
\end{Remark}
We have just seen that for every series-parallel ideal $I$ we have a tree $\P_I$ such that $\P$ is the support poset of $I$. Our next results states that the converse is also true.

\begin{Proposition}\label{prop:spIdeal}
Let $\P$ be a forest. There is a series-parallel ideal $I_\P$ such that $\P$ is its support poset.
\end{Proposition} 
\begin{proof}
Let $\P$ be a tree whose root is $r$. For every leaf $i$ of $\P$ let $I_i=\langle x_i\rangle$. For every inner node $j$ whose children are $j_1,\dots,j_k$ let $I_j=\langle x_j\rangle\cap\sum_{i=1}^{i=k}I_{j_k}$. At each stage we have that the support poset of $I_j$ is the upper set of $j$, $\P_{\geq j}$, hence we have that $I_\P=I_r$.
\end{proof}

Observe that the ideal we construct in Proposition \ref{prop:spIdeal} is in fact the {\em leaf} ideal of Proposition \ref{prop:leafIdeal}. The proof is an easy inspection of each of the generators. While we always have that $\P_{I(\P)}=\P$ it is not always the case that $I=I(\P_I)$ as the following example shows.

\begin{Example}\label{ex:differentSeriesParallel}
Consider the system $S_1$ in Figure \ref{fig:spsystem1}. It is a series-parallel system whose construction procedure following Theorem \ref{th:spIdeal} yields the cut ideal 
\[
I_{S_1}=\left( \langle x_1\rangle \cap \left(\langle x_2\rangle \cap \langle x_3\rangle+\langle x_4\rangle\right)\right)\cap \left( \langle x_5\rangle \cap \left(\langle x_6\rangle \cap \langle x_7\rangle+\langle x_8\rangle\right)\right)\subseteq\kb[x_1,\dots,x_8].
\]

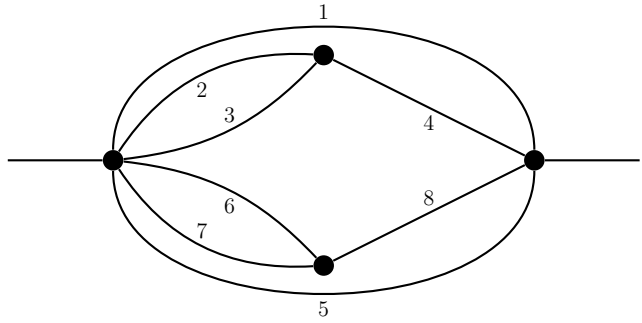
\begin{figure}
\centering
\begin{tikzpicture}[scale=0.7,transform shape]
\draw
node at (0,0) [fill,circle](n1){}
node at (4,2) [fill,circle](n2){}
node at (4,-2) [fill,circle](n3){}
node at (8,0) [fill,circle](n4){};


\draw (-2,0) edge[thick](n1);
\draw (n1) edge[bend right=20,thick] node[midway,above] {$3$} (n2) ;
\draw (n1) edge[bend left=30,thick]  node[midway,below] {$2$}(n2);
\draw (n1) edge[bend right=30,thick] node[midway,above] {$7$} (n3);
\draw (n1) edge[bend left=20,thick]node[midway,below] {$6$} (n3);
\draw (n2) edge[thick] node[midway,below] {$4$}(n4);
\draw (n3) edge[thick] node[midway,above] {$8$}(n4);
\draw (n1) edge[bend left=90,thick] node[midway,above] {$1$} (n4);
\draw (n1) edge[bend right=90,thick] node[midway,below] {$5$}(n4);
\draw (n4) edge[thick](10,0);
\end{tikzpicture}
\caption{$S_1$, a series-parallel system.}
\label{fig:spsystem1}
\end{figure}

We have that $I_{S_1}=\langle x_1x_2x_3x_5x_6x_7,x_1x_2x_3x_5x_8,x_1x_4x_5x_6x_7,x_1x_4x_5x_8\rangle$ and its support poset $\P_{I_{S_1}}$ is given by the sets
\[
C_1=\{1,5\},\,C_2=\{1,2,3,5\},\,C_3=\{1,2,3,5\},\,C_4=\{1,4,5\},
\]
\[
C_5=\{1,5\},\,C_6=\{1,5,6,7\},\,C_7=\{1,5,6,7\},\,C_8=\{1,5,8\}
\]
whose Hasse diagram (for $x_1<x_5,x_2<x_3$ and $x_6<x_7$) is depicted in Figure \ref{fig:posetSP}

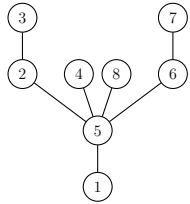
\begin{figure}[h]
\centering

      \begin{tikzpicture}[scale=.5, transform shape, vertices/.style={text width= 1.5em,align=center,draw=black, fill=white, circle, inner sep=1pt}]
             \node [vertices] (0) at (3,0){1};
             \node [vertices] (1) at (3,1.5){5};
             \node [vertices] (2) at (1,3){2};
             \node [vertices] (3) at (2.5,3){4};
             \node [vertices] (4) at (3.5,3){8};
             \node [vertices] (5) at (5,3){6};
             \node [vertices] (6) at (1,4.5){3};
             \node [vertices] (7) at (5,4.5){7};

              \foreach \to/\from in {0/1, 1/2, 1/3, 1/4, 1/5, 2/6, 5/7}
     \draw [-] (\to)--(\from);		
     \end{tikzpicture}     
     \caption{Support poset of $I_{S_1}$}
     \label{fig:posetSP}
\end{figure}

Observe that following the procedure in Proposition \ref{prop:spIdeal} on $\P_{I_{S_1}}$ we obtain the series-parallel ideal  $I(\P_{I_{S_1}})=\langle x_1x_2x_3x_5,x_1x_4x_5,x_1x_5x_8,x_1x_5x_6x_7\rangle$ which is also the {\em leaf ideal} of $\P_{I_{S_1}}$. This is the cut ideal of the series-parallel system $S_2$ in Figure \ref{fig:spsystem2}. Observe that $S_1$ and $S_2$ are two different series-parallel systems, yet their respective cut ideals have the same support poset.

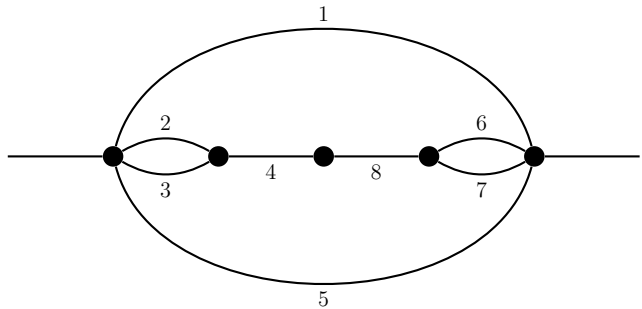
\begin{figure}
\centering
\begin{tikzpicture}[scale=0.7,transform shape]
\draw
node at (0,0) [fill,circle](n1){}
node at (2,0) [fill,circle](n2){}
node at (4,0) [fill,circle](n3){}
node at (6,0) [fill,circle](n4){}
node at (8,0) [fill,circle](n5){};


\draw (-2,0) edge[thick](n1);
\draw (n1) edge[bend right=30,thick] node[midway,below] {$3$} (n2) ;
\draw (n1) edge[bend left=30,thick]  node[midway,above] {$2$}(n2);
\draw (n1) edge[bend right=75,thick] node[midway,below] {$5$} (n5);
\draw (n1) edge[bend left=75,thick]node[midway,above] {$1$} (n5);
\draw (n2) edge[thick] node[midway,below] {$4$}(n3);
\draw (n3) edge[thick] node[midway,below] {$8$}(n4);
\draw (n4) edge[bend left=30,thick] node[midway,above] {$6$} (n5);
\draw (n4) edge[bend right=30,thick] node[midway,below] {$7$}(n5);
\draw (n5) edge[thick](10,0);
\end{tikzpicture}
\caption{$S_2$, a series-parallel system.}
\label{fig:spsystem2}
\end{figure}

\end{Example}

\section{Open questions}
The fact that not every poset is the support poset of a monomial ideal poses a question that we have partially answered in this paper, namely
\begin{Question}
How can we characterize those posets that are realizable as the support poset of a monomial ideal?
\end{Question}
We have seen that forests and collections of disjoint diamonds are in the category of realizable posets. A different question is about those posets that cannot be realized as support posets:
\begin{Question}
How can we characterize those posets that are not realizable as the support poset of a monomial ideal?
\end{Question}
Finally, a big open question arises from the fact that a given poset can be seen a the support poset of several ideals which have different properties, cf. Remark \ref{rm:differentIdeals}. Even within the class of forests and series-parallel ideals we have seen that a given tree can be the support poset of different series-parallel ideals, cf. Example \ref{ex:differentSeriesParallel}, these may have different number of generators, hence different Betti numbers, projective dimension, regularity, etc. A wide open question is then to find relations between ideals having the same support poset and properties of the ideal that can be read off the support poset.
\begin{Question}
What properties are shared by ideals that have the same support poset?
\end{Question}
Attempts to answer these questions will prove the usefulness of the support poset as a tool to study monomial ideals and their structure.

\section*{Acknowledgements}
The authors are partially funded by grant MTM2017-88804-P of Ministerio de Econom\'ia, Industria y Competitividad (Spain). The first cited author is partially funded by University of La Rioja predoctoral (education) researcher grant on 2018.

{}
\end{document}